\documentclass{elsarticle}
\usepackage{enumerate}
\usepackage{amssymb}
\usepackage{amsmath}
\usepackage{amsfonts}
\usepackage{amsthm}
\usepackage{color}
\usepackage{mathrsfs}
\usepackage{amsmath}
\usepackage{tikz}
\usepackage[colorlinks=true, pdfstartview=FitV, linkcolor=blue, citecolor=blue, urlcolor=blue]{hyperref}
\usepackage{doi}
\usepackage{url}

\newtheorem{theorem}{Theorem}[section]

\newtheorem{definition}[theorem]{Definition}
\newtheorem{lemma}[theorem]{Lemma}

\newtheorem{remark}[theorem]{Remark}

\numberwithin{equation}{section}

\begin{document}

\title{Weighted estimates for the multilinear maximal function on the upper half-spaces}

\author[chay]{Wei Chen}
\ead{weichen@yzu.edu.cn}
\address[chay]{School of Mathematical Sciences, Yangzhou University, Yangzhou 225002, China}

\author[chay1]{Chunxiang Zhu}
\ead{cxzhu\_yzu@163.com}
\address[chay1]{School of Mathematical Sciences, Yangzhou University, Yangzhou 225002, China}

\begin{abstract}
For a general dyadic grid, we give a Calder\'{o}n-Zygmund type decomposition, which is the principle fact about the multilinear maximal function $\mathfrak{M}$ on the upper half-spaces. Using the decomposition, we study the boundedness of $\mathfrak{M}.$ We obtain a natural extension to the multilinear setting of Muckenhoupt's
weak-type characterization. We also partially obtain characterizations of Muckenhoupt's strong-type inequalities
with one weight.
Assuming the reverse H\"{o}lder's condition, we get a multilinear analogue of Sawyer's two weight theorem. Moreover, we also get Hyt\"{o}nen-P\'{e}rez type weighted estimates.
\end{abstract}

\begin{keyword}Upper half-space, Multilinear maximal function, Weighted inequality, Reverse H\"{o}lder's condition, Hyt\"{o}nen-P\'{e}rez type estimate.
\MSC Primary: 42B25; Secondary: 42B20, 42B35.
\end{keyword}

\maketitle

\section{Introduction}\label{Tntro}
\subsection{Hardy-littlewood maximal function on $\mathbb{R}^n$}\label{Tntro-1}
Let $\mathbb R^n$ be the $n\hbox{-dimensional}$ real Euclidean
space and $f$ a real valued measurable function, the classical Hardy-littlewood maximal function is defined by
$$Mf(x)=\sup\limits_{x\in Q}\frac{1}{|Q|}\int_Q|f(y)|dy,$$
 where $Q$ is a cube with its sides parallel to the coordinate
axes and $|Q|$ is the Lebesgue measure of $Q.$

A weight will be a nonnegative locally integrable function.   
Let $u,~v$ be two weights. Muckenhoupt \cite{Muckenhoupt B} showed that
$$\left\{
  \begin{array}{ll}
    M:L^p(v,\mathbb{R}^n)\rightarrow L^{p,\infty}(u,\mathbb{R}^n)&\hbox{ iff }(u,v)\in A_p, \hbox{ where } p\geq1;\\
    M:L^p(v,\mathbb{R}^n)\rightarrow L^{p}(v,\mathbb{R}^n)&\hbox{ iff }v\in A_p, \hbox{ where } p>1
  \end{array}
\right.$$
Let $p>1,$ Sawyer \cite{Sawyer E T.} gave the
testing condition and characterized the weights for which $M$ is bounded from $L^p(v,\mathbb{R}^n)$ to $L^{p}(u,\mathbb{R}^n).$
Motivated by \cite{Muckenhoupt B, Sawyer E T.},
the theory of weighted inequalities developed rapidly, not only for the Hardy--Littlewood maximal operator but also for some of the main operators in Harmonic Analysis like Calder\'on--Zygmund operators (see \cite{Garcia Rubio}
and \cite{Cruz-Uribe D. J. M. Martell} for more information).

Recently, a large body of literature on the topic of multilinear weighted norm inequalities appeared. This study is based on multiple simultaneous
decompositions and is naturally more complicated than its linear counterpart,
but is also more far-reaching and yields more flexible results. Weighted estimates for the maximal operator $\prod_{j=1}^m Mf_j$ ($m$-fold product of $M$) in the multilinear setting were studied in \cite{GT} and \cite{PT}.
The new multilinear maximal function
\begin{equation*}\label{multi_maximal_operator}\mathcal{M}(f_1,...,f_m)(x) := \sup\limits_{x\in Q}
\prod\limits_{i=1}\limits^{m}\frac{1}{|Q|}\int_Q|f_i(y_i)|dy_i, \quad x\in \mathbb R^n
\end{equation*}
associated with cubes with sides parallel to the coordinate
axes was first defined and the corresponding weight theory was studied in \cite{Lerner A.K. Ombrosi S.}.
The importance of this operator is that it is strictly smaller than the $m$-fold product of $M$. Moreover, it generalizes the Hardy--Littlewood
maximal function (case $m=1$) and in several ways it controls the class
of multilinear Calder\'{o}n--Zygmund operators as shown in \cite{Lerner A.K. Ombrosi S.}.
The relevant class of multiple weights for $\mathcal{M}$ is given by the condition $A_{\overrightarrow{p}}$
\cite[Definition 3.5]{Lerner A.K. Ombrosi S.}.
The more general case was extensively discussed in
\cite{Grafakos L. Liu L. G. Perez C. Torres R. H.,Grafakos L. Liu L. G. Yang D. C.}.
Using a dyadic discretization technique, Dami\'{a}n, Lerner and P\'{e}rez \cite{W. Damian}
and Li, Moen and Sun \cite{Li Moen Sun} proved some sharp weighted norm inequalities for the multilinear maximal operator $\mathcal M.$
In order to establish the generalization of Sawyer's theorem to the multilinear setting, Chen and Dami\'{a}n \cite{Chen-Damian} introduced
a reverse H\"{o}lder's condition $RH_{\overrightarrow{p}}$ on the weights and established the multilinear version of Sawyer's result; however the method do not work without $RH_{\overrightarrow{p}}.$
Later on, the condition $RH_{\overrightarrow{p}}$ was used in \cite{cao-xue,cao-xue1,CL1,Cruz-Moen,Sehba}. Recently, Cruz-Uribe and Moen \cite{Cruz-Moen} proved a multilinear version of the reverse H\"{o}lder's inequality
in the theory of Muckenhoupt $A_p$ weights. 
In our opinion, it is difficult to establish the multilinear version of Sawyer's result without any assumptions. In fact, we also found that Li, Xue and Yan \cite{W. M. Li} introduced a kind of monotone property and established the multilinear version of Sawyer's result.
Note that if $v=\prod_{i=1}^m\omega_i^{{p}/{p_i}},$ then the condition
$(v,\overrightarrow{\omega})\in A_{\overrightarrow{p}}$ implies the reverse H\"{o}lder's condition $\overrightarrow{\omega}\in RH_{\overrightarrow{p}}$ \cite[Proposition 2.3]{cao-xue}. 
In addition, Chen and Dami\'{a}n investigated a bound $B_{\overrightarrow{p}}$ \cite[Theorem 2]{Chen-Damian} and a mixed
bound $A_{\overrightarrow{p}}-W^{\infty}_{\overrightarrow{p}}$ \cite[Theorem 3]{Chen-Damian} for the multilinear maximal operator, which are the multilinear versions of Hyt\"{o}nen-P\'{e}rez type weighted estimates \cite[Theorem 4.3]{T. Hytonen}.

\subsection{Maximal function on the upper half-space $\mathbb{R}_{+}^{n+1}$}\label{Tntro-2}
Given a function $f$ on $\mathbb{R}^n,$ we define a
maximal function on the upper half-space $\mathbb{R}_{+}^{n+1}=\{(x,t)\,:\,x\in\mathbb{R}^{n},t\geq0\}$
by setting
$$
\widetilde{M}f(x,t)=\sup_{x\in Q,\,l(Q)\geq t}\frac{1}{|Q|}\int_{Q}|f(x)|dx,
$$
where $Q$ is a cube with its sides parallel to the coordinate
axes and $|Q|$ is the Lebesgue measure of $Q.$
The maximal function controls the Poisson integral
$$
Pf(x,t)=\int_{\mathbb{R}^{n}}f(y)P(x-y,t)dy\quad x\in\mathbb{R}^{n},\,t\geq0
$$
where
$$
P(x,t)=c_{n}\frac{t}{\left(|x|^{2}+t^{2}\right)^{\frac{n+1}{2}}}
$$
is the Poisson Kernel.

Let $\mu$ be a measure on $\mathbb{R}^{n+1}_+$ and $v$ a
weight on $\mathbb{R}^n.$ Carleson \cite{Carleson} characterized the positive Borel
measures $\mu$ on $\mathbb{R}_{+}^{n+1}$ such that $\mathcal{M}$ is of
strong type $(p,p)$ for $p>1$ and of weak type $(1,1)$. Later on, Fefferman and Stein \cite{FS}
found a condition on the pair
$(\mu,v)$ to be sufficient for the boundedness of the maximal operator $\widetilde{M}$ from
$L^p(\mathbb{R}^n,v)$ into $L^p(\mathbb{R}^{n+1}_+,\mu)$ for $p>1$ and from
$L^1(\mathbb{R}^n,v)$ into $L^{1,\infty}(\mathbb{R}^{n+1}_+,\mu).$
Let $p>1,$ Ruiz \cite{Ruiz}
and Ruiz and Torrea \cite{Ruiz2} obtained the exact conditions on the pair
$(\mu,v)$ for maximal operator $\widetilde{M}$ to be a bounded operator from
$L^p(\mathbb{R}^n,v)$ into $L^{p,\infty}(\mathbb{R}^{n+1}_+,\mu)$
and from $L^p(\mathbb{R}^n,v)$ into $L^p(\mathbb{R}^{n+1}_+,\mu),$ respectively.
In \cite{P. Byung-Oh} and \cite{J. Garcia-Cuerva}, the analogues of the above
results have been developed in spaces of $\hbox{(non-)}$homogeneous
type. In addition, the operator $\widetilde{M}$ can be defined in martingale space, and
the weighted inequalities also have their martingale
versions \cite{W. Chen}. Recently, Rivera-R\'{i}os \cite{Ri}
studied quantitative versions of weighted estimates obtained by Ruiz \cite{Ruiz}
and Ruiz and Torrea \cite{Ruiz2}.

The aim of this paper is to give some multilinear analogues of
the above mentioned results for the maximal function on the upper half-space $\mathbb{R}_{+}^{n+1}.$ Given $\overrightarrow f=(f_1,\dots,f_m)$, we define the multilinear maximal operator $\mathfrak{M}$ on the upper half-space $\mathbb{R}_{+}^{n+1}$ by $$\mathfrak{M}(\overrightarrow f\,)(x,t)=\sup_{x\in Q,\,l(Q)\geq t}\prod_{i=1}^m\frac{1}{|Q|}\int_Q|f_i(y_i)|dy_i$$
where $Q$ is a cube with its sides parallel to the coordinate
axes and $|Q|$ is the Lebesgue measure of $Q.$
We provide some weighted estimates for $\mathfrak{M}.$
Our approaches are mainly based on a Calder\'{o}n-Zygmund type decomposition suited to the multilinear setting and the multilinear Carleson embedding theorem \cite[Lemma 3]{Chen-Damian}.

Firstly, we obtain a natural extension to the multilinear setting of Muckenhoupt's
weak-type characterization. All unexplained notations can be found in Section \ref{sec2}.

\begin{theorem}\label{theorem_Ap} Let $\overrightarrow{P}=(p_1,\cdots,p_m)$ with $1< p_1,\cdots,p_m<\infty$ and $1/{p_1}+\cdots+1/{p_m}=1/p$.
Let $\mu$ be a Borel measure
on $\mathbb{R}_{+}^{n+1}.$ Let $\omega_1,\cdot\cdot\cdot,\omega_m$ be weights in $\mathbb{R}^{n}.$ Then the following statements
are equivalent:
  \begin{enumerate}[\rm(1)]
    \item \label{theorem_Ap_3}$(\mu,\overrightarrow{w})$ satisfies the multilinear $A'_{\overrightarrow{P}}$ condition;
    \item \label{theorem_Ap_1} There exists a positive constant $C$ such that
        $$\mu(\widetilde{Q})^{\frac{1}{p}}\prod\limits_{i=1}^{m}\left(\frac{1}{|Q|}\int_Qf_i(y)dy\right)\leq
        C\prod\limits^{m}_{i=1}\|f_i\chi_Q\|_{L^{p_i}(\mathbb{R}^n,\omega_i)},$$
    for any $\overrightarrow{f}\in \prod\limits_{i=1}^m{L^{p_i}(\mathbb{R}^n,\omega_i)}$ and any cube $Q$ in $\mathbb{R}^n;$
    \item \label{theorem_Ap_2} There exists a positive constant $C$ such that
       $$\lambda \mu\left(\{(x,t)\in\mathbb{R}_+^{n+1}:\mathfrak{M}(\overrightarrow{f})\geq\lambda\}\right)^{\frac{1}{p}}\leq C\prod\limits^{m}_{i=1}\|f_i\|_{L^{p_i}(\mathbb{R}^n,\omega_i)},$$
     for any $\overrightarrow{f}\in \prod\limits_{i=1}^{m}{L^{p_i}(\mathbb{R}^n,\omega_i)}$ and $\lambda>0;$
    \item \label{theorem_Ap_21} There exists a positive constant $C$ such that
       $$\lambda \mu\left(\{(x,t)\in\mathbb{R}_+^{n+1}:\mathfrak{M}(\overrightarrow{f})>\lambda\}\right)^{\frac{1}{p}}\leq C\prod\limits^{m}_{i=1}\|f_i\|_{L^{p_i}(\mathbb{R}^n,\omega_i)},$$
     for any $\overrightarrow{f}\in \prod\limits_{i=1}^{m}{L^{p_i}(\mathbb{R}^n,\omega_i)}$ and $\lambda>0.$
  \end{enumerate}
  Moreover, if we denote the smallest constants $C$ in \eqref{theorem_Ap_1}, \eqref{theorem_Ap_2} and \eqref{theorem_Ap_21}
by $[v,\overrightarrow{\omega}]'_{A'_{\overrightarrow{p}}},$ $\|\mathfrak{M}\|'$ and $\|\mathfrak{M}\|,$ respectively,
then we have
$$[v,\overrightarrow{\omega}]_{A'_{\overrightarrow{p}}}=[v,\overrightarrow{\omega}]'_{A'_{\overrightarrow{p}}},$$
$$\|\mathfrak{M}\|'=\|\mathfrak{M}\|,$$
$$[v,\overrightarrow{\omega}]'_{A'_{\overrightarrow{p}}}\leq\|\mathfrak{M}\|'$$
and $$\|\mathfrak{M}\|\lesssim [v,\overrightarrow{\omega}]'_{A_{\overrightarrow{p}}}.$$
  \end{theorem}

Secondly, we partially obtain characterizations of Muckenhoupt's strong-type inequalities
with one weight. There are two different versions. The first is Theorem \ref{thm:m} and the other is Theorem \ref{prop s}.

\begin{theorem}\label{thm:m}
Suppose $1<p_1,\ldots,p_m<\infty$, $1/p=1/{p_1}+\cdots+1/{p_m}$, $\overrightarrow{w}\in A'_{\overrightarrow P}$
and $(\mu,v_{\overrightarrow{w}})\in C_0.$  Then
\begin{equation}\label{eq:emm}
  \|\mathfrak{M}(\overrightarrow{f})\|_{L^p(\mathbb{R}_+^{n+1},\mu)}\lesssim [\mu,v_{\overrightarrow{w}}]^{1/p}_{ C_0}
[\overrightarrow{w}]^{\bar{p}}_{A'_{\overrightarrow P}}\prod_{i=1}^m \|f_i\|_{L^{p_i}(\mathbb{R}^n,w_i)},
\end{equation}
where $\bar{p}=\max\{p_1',\cdots,p_m'\}.$
\end{theorem}

\begin{theorem}\label{prop s}
Suppose $1<p_1,\ldots,p_m<\infty$, $1/p=1/{p_1}+\cdots+1/{p_m}$, $\overrightarrow{w}\in A_{\overrightarrow P}$
and $(\mu,v_{\overrightarrow{w}})\in C_{\infty}.$  Then
\begin{equation*}
  \|\mathfrak{M}(\overrightarrow{f})\|_{L^p(\mathbb{R}_+^{n+1},\mu)}\lesssim [\mu,v_{\overrightarrow{w}}]^{1/p}_{ C_\infty}
[\overrightarrow{w}]^{\bar{p}}_{A_{\overrightarrow P}}\prod_{i=1}^m \|f_i\|_{L^{p_i}(\mathbb{R}^n,w_i)},
\end{equation*}
where $\bar{p}=\max\{p_1',\cdots,p_m'\}.$
\end{theorem}

Thirdly, assuming the reverse H\"{o}lder's condition, we get a multilinear analogue of Sawyer's two weight theorem.

\begin{theorem}\label{theorem_Sp}Suppose $1<p_1,\ldots,p_m<\infty$, $1/p=1/{p_1}+\cdots+1/{p_m}.$
If $(\omega_1, ~\omega_2,\cdot\cdot\cdot,~\omega_m)\in RH_{\overrightarrow{p}},$
then the following statements
are equivalent:\begin{enumerate}[\rm(1)]
\item \label{thm A 1}There exists a positive constant $C$ such that
\begin{equation*}
\|\mathfrak{M}(\overrightarrow{f})\|_{L^p(\mathbb{R}_+^{n+1},\mu)}\leq
C\prod\limits^m_{i=1}\|f_i\|_{L^{p_i}(\mathbb{R}^n,\omega_i)},
~\forall f_i\in L^{p_i}(\mathbb{R}^n,\omega_i);
\end{equation*}
\item \label{thm A 2} There exists a positive constant $C$ such that
\begin{equation}
\label{Th_change}\|\mathfrak{M}(\overrightarrow{f\sigma})\|_{L^p(\mathbb{R}_+^{n+1},\mu)}\leq
C\prod\limits^m_{i=1}\|f_i\|_{L^{p_i}(\mathbb{R}^n,\sigma_i)},
~\forall f_i\in L^{p_i}(\mathbb{R}^n,\sigma_i);
\end{equation}
\item $(\mu,\overrightarrow{\omega})$
satisfies the condition $S_{\overrightarrow{p}}.$
\end{enumerate}
Moreover, we denote the smallest constants $C$ in (\ref{thm A 1})
and (\ref{thm A 2}) by $\|\mathfrak{M}\|$ and $\|\mathfrak{M}\|',$ respectively.
Then it follows that
$$[v,\overrightarrow{\omega}]_{S_{\overrightarrow{p}}}
\leq\|\mathfrak{M}\|=\|\mathfrak{M}\|'\lesssim
[\overrightarrow{\omega}]_{RH_{\overrightarrow{p}}}^{\frac{1}{p}}[v,\overrightarrow{\omega}]_{S_{\overrightarrow{p}}}.$$
\end{theorem}

Finally, we give Hyt\"{o}nen-P\'{e}rez type weighted estimates.

\begin{theorem}\label{theorem_Bp}Suppose $1<p_1,\ldots,p_m<\infty$,
$1/p=1/{p_1}+\cdots+1/{p_m}.$
If $(v,\overrightarrow{\omega})\in B'_{\overrightarrow{p}},$ then the following statements
are valid:\begin{enumerate}[\rm(1)]
\item \label{thm B 1}There exists a positive constant $C$ such that
$$
\|\mathfrak{M}(\overrightarrow{f})\|_{L^p(\mathbb{R}^{n+1}_+, \mu)}\leq
C\prod\limits^m_{i=1}\|f_i\|_{L^{p_i}(\mathbb{R}^{n},\omega_i)},
~\forall f_i\in L^{p_i}(\mathbb{R}^{n},\omega_i);
$$
\item \label{thm B 2}There exists a positive constant $C$ such that
$$
\|\mathfrak{M}(\overrightarrow{f\sigma})\|_{L^p(\mathbb{R}^{n+1}_+, \mu)}\leq
C\prod\limits^m_{i=1}\|f_i\|_{L^{p_i}(\mathbb{R}^{n},\sigma_i)},
~\forall f_i\in L^{p_i}(\mathbb{R}^{n},\sigma_i).
$$
\end{enumerate}
Moreover, we denote the smallest constants $C$ in (\ref{thm B 1})
and (\ref{thm B 2}) by $\|\mathfrak{M}\|$ and $\|\mathfrak{M}\|',$ respectively.
Then it follows that $$\|\mathfrak{M}\|=\|\mathfrak{M}\|'\lesssim[\mu,\overrightarrow{\omega}]_{B'_{\overrightarrow{p}}}.$$
\end{theorem}

\begin{theorem}\label{theorem_bi A Fujii}Suppose $1<p_1,\ldots,p_m<\infty$,
$1/p=1/{p_1}+\cdots+1/{p_m}.$
If $(\mu,\overrightarrow{\omega})\in A'_{\overrightarrow{p}}$
and $\overrightarrow{\omega}\in W^{\infty}_{\overrightarrow{p}},$ then the following statements
are valid:\begin{enumerate}[\rm(1)]
\item \label{thm C 1}There exists a positive constant $C$ such that
$$
\|\mathfrak{M}(\overrightarrow{f})\|_{L^p(\mathbb{R}^{n+1}_+, \mu)}\leq
C\prod\limits^m_{i=1}\|f_i\|_{L^{p_i}(\mathbb{R}^{n},\omega_i)},
~\forall f_i\in L^{p_i}(\mathbb{R}^{n},\omega_i);
$$

\item \label{thm C 2}There exists a positive constant $C$ such that
$$
\|\mathfrak{M}(\overrightarrow{f\sigma})\|_{L^p(\mathbb{R}^{n+1}_+, \mu)}\leq
C\prod\limits^m_{i=1}\|f_i\|_{L^{p_i}(\mathbb{R}^{n},\sigma_i)},
~\forall f_i\in L^{p_i}(\mathbb{R}^{n},\sigma_i).
$$
\end{enumerate}
Moreover, we denote the smallest constants $C$ in (\ref{thm C 1})
and (\ref{thm C 2}) by $\|\mathfrak{M}\|$ and $\|\mathfrak{M}\|',$ respectively.
Then it follows that $\|\mathfrak{M}\|=\|\mathfrak{M}\|'
\lesssim [\mu,\overrightarrow{\omega}]_{A'_{\overrightarrow{p}}}
[\overrightarrow{\omega}]_{W_{\overrightarrow{p}}^\infty}^{\frac{1}{p}}.$
\end{theorem}

\begin{remark}\label{reduce}In the above theorems, we can obviously assume that $f_i\geq0$ and
$f_i\in L^1(\mathbb{R}^n),$ $i=1,\cdot\cdot\cdot,m.$ Indeed, there are integral functions
$f^{(j)}_i,$ such that $f^{(j)}_i\uparrow f_i,$ $i=1,\cdot\cdot\cdot,m.$ It is clear that $\{(x,t)\in \mathbb{R}_+^{n+1} : \mathfrak{M}\overrightarrow{f}(x,t)>\lambda\}=\bigcup_j\{(x,t)\in \mathbb{R}_+^{n+1}:
\mathfrak{M}\overrightarrow{f^{(j)}}(x,t)>\lambda\},$ where $\overrightarrow{f^{(j)}}=(f^{(j)}_1,\dots,f^{(j)}_m).$
\end{remark}

The article is organized as follows.
In Section \ref{sec2}, we state some preliminaries and definitions
and give a Calder\'{o}n-Zygmund type decomposition. In Section \ref{proofs}, we provide the proofs of the above theorems.

Throughout the paper, we use the notation $A\lesssim B$ to indicate that there is a constant $c$,
independent of the weight constant, such that $A\leq cB.$ We write $A\thickapprox B$ when $A\lesssim B$
and $B\lesssim A.$

\section{Preliminaries and definitions} \label{sec2}

Let $Q$ be a cube in $\mathbb{R}^n.$ We denote by $\tilde{Q}$ the cube built as follows
$$
\widetilde{Q}=\left\{ (x,t)\in\mathbb{R}_{+}^{n+1}\,:\,x\in Q,\text{ and }0\leq t<l(Q)\right\} ,
$$
in other words, $\widetilde{Q}$ is the cube in $\mathbb{R}_{+}^{n+1}$
having $Q$ as a face.

Recall that the standard dyadic grid in $\mathbb{R}^n$ consists of the cubes
$$2^{-k}([0,1)^n+j),k\in Z,j\in \mathbb{Z}^n.$$
Denote the standard grid by $\mathcal{D}.$

By a general dyadic grid $\mathfrak{D}$ we mean a collection of cubes
with the following properties: \begin{enumerate}[\rm(i)]
\item for any $Q\in \mathfrak{D}$ its side length $l_Q$ is of the form $2^k, k\in Z;$
\item $Q\cap R\in \{Q,R,\emptyset\}$ for any $Q,R\in\mathfrak{D};$
\item the cubes of a fixed side length $2^k$ form a partition of $\mathbb{R}^n.$ \end{enumerate}

We say the $\mathfrak{S}\triangleq\bigcup\limits_{k,j}\{Q_{j}^{k}\}$ is a sparse family of cubes if:
\begin{enumerate}[\rm(i)]
\item the cubes $Q_j^k$ are disjoint in $j,$ with $k$ fixed;
\item denote $\Omega_k=\cup_jQ_j^k,$ then $\Omega_{k+1}\subseteq\Omega_k;$
\item $|\Omega_{k+1}\cap Q_j^k|\leq\frac{1}{2}|Q_j^k|.$
\end{enumerate}

With each sparse family $\{Q_j^k\}$ we associate the sets $E_{Q_j^k}\triangleq Q_j^k\backslash\Omega_{k+1}.$
For simplicity, we also denote $E_{Q_j^k}$ by $E_j^k$.
Observe that the sets $\{E_j^k\}$ are pairwise disjoint and $|E_j^k|\leq|Q_j^k|\leq2|E_j^k|.$

The following Lemma \ref{Lemma:CZ} is a Calder\'{o}n-Zygmund type decomposition which help us to prove our theorems in a unified approach.

\begin{lemma}\label{Lemma:CZ} Let $a=2^{m(n+1)}.$ Let $\mathfrak{D}$ be a general dyadic grid. Let $f_1,f_2,\cdot\cdot\cdot,f_m$ be non-negative integrable functions.
For each $k\in\mathbb{Z},$ we can choose a family $\{Q_{j}^{k}\}_{j\in J_k}\subseteq\mathfrak{D}$ such that

\begin{enumerate}[(1)]
\item \label{lemde1}$a^k<\prod\limits_{i=1}^{m}\frac{1}{|Q_{j}^{k}|}\int_{Q_{j}^{k}}f_i(y)dy\leq 2^{mn}a^k.$
\item \label{lemde2}The interiors of $\widetilde{Q_{j}^{k}}$ with $j\in J_k$ are pairwise disjoint.
\item \label{lemde3}$\widehat{\Omega_k}\triangleq\left\{(x,t)\in \mathbb{R}_+^{n+1} : \mathfrak{M}^{\mathfrak{D}}\overrightarrow{f}(x,t)>a^k\right\}=\bigcup_{j\in J_{k}}\widetilde{Q_{j}^{k}}.$
\end{enumerate}
Furthermore, the family of cubes $\mathfrak{S}\triangleq\bigcup\limits_{k,j}\{Q_{j}^{k}\}$ is sparse.\end{lemma}

With each sparse family $\{Q_j^k\}$ we associate the sets
$\widehat{E_{Q_j^k}}\triangleq \widetilde{Q_j^k}\backslash\widehat{{\Omega}_{k+1}}.$
For simplicity, we also denote
$\widehat{E_{Q_j^k}}$ by
$\widehat{E_j^k}.$

In this paper, we will use the well-known Lemma \ref{Hytonen} from \cite{T. Hytonen}.

\begin{lemma}\label{Hytonen}There are $2^n$ dyadic grids $\mathfrak{D}_\beta$ such that for
all cube $Q\subseteq \mathbb{R}^n$ there exists a cube $Q_\beta\in\mathfrak{D}_\beta$ such that
$Q\subseteq Q_\beta$ and $l_{Q_\beta}\leq 6l_Q .$
\end{lemma}

\begin{remark}\label{decop}
It follows from Lemma \ref{Hytonen} that there exists $2^n$ families of dyadic grids $\mathfrak{D}_\beta$ such that
$$\mathfrak{M}(\overrightarrow{f})(x,t)\leq 6^{mn}\sum_{\beta=1}^{2^n}\mathfrak{M}^{\mathfrak{D}_\beta}(\overrightarrow{f})(x,t),$$
where
$$\mathfrak{M}^{\mathfrak{D}_\beta}(\overrightarrow{f})(x,t)=\sup_{x\in Q\in\mathfrak{D}_\beta,l(Q)\geq t}\prod_{i=1}^m\frac{1}{|Q|}\int_{Q}f_i(y)dy.$$
Let $q>0.$ It follows that
$$\|\mathfrak{M}(\overrightarrow{f})\|_{L^{p,\infty}(\mathbb{R}_+^{n+1},\mu)}\lesssim
\sum_{\beta=1}^{2^n}\|\mathfrak{M}^{\mathfrak{D}_\beta}(\overrightarrow{f})\|_{L^{p,\infty}(\mathbb{R}_+^{n+1},\mu)}$$
and
$$\|\mathfrak{M}(\overrightarrow{f})\|_{L^{p}(\mathbb{R}_+^{n+1},\mu)}\lesssim
\sum_{\beta=1}^{2^n}\|\mathfrak{M}^{\mathfrak{D}_\beta}(\overrightarrow{f})\|_{L^{p}(\mathbb{R}_+^{n+1},\mu)}.$$
Let $\mathfrak{D}$ be a general dyadic grid. It follows from Lemma \ref{Lemma:CZ} that there exists a sparse subset $\mathfrak{S}\subset \mathfrak{D}$ such that
\begin{eqnarray*}\mathfrak{M}^{\mathfrak{D}}(\overrightarrow{f})^q
&\leq& \sum_{k\in\mathbb{Z}}\sum_{j\in J_k}a^{(k+1)q}\chi_{\widehat{E_j^k}}\leq a^q\sum_{Q\in \mathfrak{S}}\big(\prod\limits_{i=1}^{m}\frac{1}{|Q|}\int_{Q}f_i(y)dy\big)^q\chi_{\widehat{E_Q}}\\
&\leq& a^q\sum_{Q\in \mathfrak{S}}\big(\prod\limits_{i=1}^{m}\frac{1}{|Q|}\int_{Q}f_i(y)dy\big)^q\chi_{\widetilde{Q}}.
\end{eqnarray*}\end{remark}

In addition, there is a formulation of the Carleson embedding theorem in the multilinear
setting which was proved in \cite{Chen-Damian}.

\begin{lemma} \cite{Chen-Damian}\label{lem_Carleson_thm} Let $\mathfrak{D}$ be a general dyadic grid.
If the nonnegative numbers $a_Q$ and non-negative
function $\prod\limits_{i=1}\limits^{m}\sigma_i^{\frac{p}{p_i}}$
satisfy
\begin{equation}\label{lem_Carleson_ass}
\sum\limits_{Q\subseteq R}a_Q\leq A \int_R \prod\limits^m_{i=1}\sigma_i^{\frac{p}{p_i}}(x)d x,
~\forall R\in \mathfrak{D},\end{equation}
then
\begin{eqnarray*}\label{lem_Carleson}
\sum\limits_{Q\in \mathfrak{D}}a_Q\big(\prod\limits_{i=1}
   \limits^{m}\frac{1}{\sigma_i(Q)}\int_{Q}f_i(x)\sigma_i(x)dx\big)^p
\leq A(\prod\limits_{i=1}
       \limits^{m}p'_i)^p\prod\limits_{i=1}
       \limits^{m}\big(\int_{\mathbb{R}^n}f_i^{p_i}(x)\sigma_i(x) dx\big)^{\frac{p}{p_i}}.
\end{eqnarray*}
\end{lemma}

Now, we introduce the definitions which will be used in this paper.
First, let us pay attention to  multiple $A_p$ weights.
In \cite{Lerner A.K. Ombrosi S.}, Lerner, Ombrosi, P\'{e}rez, Torres and Trujillo-Gonz\'{a}lez introduced the theory of multiple $A_{\overrightarrow{P}}$ weights.

\begin{definition}\cite{Lerner A.K. Ombrosi S.}
Let $\overrightarrow{P}=(p_1,\cdots,p_m)$ with $1< p_1,\cdots,p_m<\infty$ and $1/{p_1}+\cdots+1/{p_m}=1/p$. Given $\overrightarrow{\omega}=(\omega_1,\cdots, \omega_m)$, set
$$v_{\overrightarrow{\omega}}=\prod_{i=1}^m \omega_i^{p/{p_i}}.$$
We say that $\overrightarrow{\omega}$ satisfies the multilinear $A_{\overrightarrow{P}}$ condition if
$$[\overrightarrow{\omega}]_{A_{\overrightarrow{P}}}:=
\sup_{Q} \left(\frac{1}{|Q|}\int_Q v_{\overrightarrow{w}}(x)dx\right)^{1/p}\prod_{i=1}^m\left( \frac{1}{|Q|}\int_Q \omega_i^{1-p_i'}(x)dx\right)^{1/{p_i'}}<\infty,$$
where the supremum is taken over all cubes in $\mathbb{R}^n$ and $[\overrightarrow{\omega}]_{A_{\overrightarrow{P}}}$ is called the $A_{\overrightarrow{P}}$ constant of $\overrightarrow{\omega}.$
\end{definition}

We define a new multiple $A_p$ weights which involve a Borel measure
on $\mathbb{R}_{+}^{n+1}.$

\begin{definition}
Let $\overrightarrow{P}=(p_1,\cdots,p_m)$ with $1< p_1,\cdots,p_m<\infty$ and $1/{p_1}+\cdots+1/{p_m}=1/p$.
Let $\mu$ be a Borel measure
on $\mathbb{R}_{+}^{n+1}.$ We denote $\overrightarrow{w}=(w_1,\cdots, w_m),$ where $w_i,$ $i=1,2,\cdots,m$ be weights in $\mathbb{R}^{n}.$
We say that $(\mu,\overrightarrow{w})$ satisfies the multilinear $A'_{\overrightarrow{P}}$ condition if
$$[\mu,\overrightarrow{w}]_{A'_{\overrightarrow{P}}}:=\sup_{Q}(\frac{\mu(\widetilde{Q})}{|Q|})^{1/p}\prod_{i=1}^m\left( \frac{1}{|Q|}\int_Q w_i^{1-p_i'}dx\right)^{1/{p_i'}}<\infty,
$$
where the supremum is taken over all cubes in $\mathbb{R}^n$ and $[\mu,\overrightarrow{w}]_{A'_{\overrightarrow{P}}}$ is called the $A'_{\overrightarrow{P}}$ constant of $(\mu,\overrightarrow{w}).$
\end{definition}

Second, we investigate the relation between $\mu(\widetilde{Q})$ and $\int_Qv(x)dx,$
where $\widetilde{Q}$ is the cube in $\mathbb{R}_{+}^{n+1}$
having $Q$ as a face, $\mu$ ie a Borel measure
on $\mathbb{R}_{+}^{n+1}$ and $v$ is a weight in $\mathbb{R}^{n}.$ Ruiz \cite{Ruiz} investigated the relation and gave the Definition \ref{def_C_infty}.

\begin{definition}\label{def_C_infty}\cite{Ruiz}
Let $\mu$ be a Borel measure
on $\mathbb{R}_{+}^{n+1}.$ Let $v$ be a weight in $\mathbb{R}^{n}.$
We say that $(\mu,v)$ satisfies the $C_{\infty}$ condition if
$$[\mu,v]_{C_{\infty}}:=\sup_{Q}\mu(\widetilde{Q})(\int_Qv(x)dx)^{-1}<\infty,
$$
where the supremum is taken over all cubes in $\mathbb{R}^n$ and $[\mu,v]_{C_{\infty}}$ is called the $C_{\infty}$ constant of $(\mu,v).$
\end{definition}

We give the following Definition \ref{def_C_0}, which is the opposite of the Definition \ref{def_C_infty}.

\begin{definition}\label{def_C_0}
Let $\mu$ be a Borel measure
on $\mathbb{R}_{+}^{n+1}.$ Let $v$ be a weight in $\mathbb{R}^{n}.$
We say that $(\mu,v)$ satisfies the $C_0$ condition if
$$[\mu,v]_{C_0}:=\sup_{Q}\mu(\widetilde{Q})^{-1}(\int_Qv(x)dx)<\infty,
$$
where the supremum is taken over all cubes in $\mathbb{R}^n$ and $[\mu,v]_{C_0}$ is called the $C_0$ constant of $(\mu,v).$
\end{definition}

Third, we define multiple $S_p$ and $B_p$ weights, which involve a Borel measure
on $\mathbb{R}_{+}^{n+1}.$

\begin{definition}
Let $\overrightarrow{P}=(p_1,\cdots,p_m)$ with $1< p_1,\cdots,p_m<\infty$ and $1/{p_1}+\cdots+1/{p_m}=1/p$.
Let $\mu$ be a Borel measure
on $\mathbb{R}_{+}^{n+1}.$ We denote $\overrightarrow{w}=(w_1,\cdots, w_m),$ where $\omega_i,$ $i=1,2,\cdots,m$ be weights in $\mathbb{R}^{n}.$
We say that $(\mu,\overrightarrow{\omega})$ satisfies the multilinear $S'_{\overrightarrow{P}}$ condition if
$$[\mu,\overrightarrow{\omega}]_{S'_{\overrightarrow{p}}}
\triangleq\sup_{Q}\big(\int_{\widetilde{Q}}\mathfrak{M}(\overrightarrow{\sigma\chi_Q})^{p}d\mu\big)^{\frac{1}{p}}
(\prod\limits^m_{i=1}\sigma_i(Q)^{\frac{1}{p_i}})^{-1}<\infty,$$
where $\overrightarrow{\sigma\chi_Q}=(\omega_1^{1-p_1'}\chi_Q,\cdots, \omega_m^{1-p_m'}\chi_Q),$ the supremum is taken over all cubes in $\mathbb{R}^n$ and $[\mu,\overrightarrow{w}]_{S'_{\overrightarrow{P}}}$ is called the $S'_{\overrightarrow{P}}$ constant of $(\mu,\overrightarrow{w}).$
\end{definition}

\begin{definition}\label{defi_B infty}
Let $\overrightarrow{P}=(p_1,\cdots,p_m)$ with $1< p_1,\cdots,p_m<\infty$ and $1/{p_1}+\cdots+1/{p_m}=1/p$. Let $\mu$ be a Borel measure
on $\mathbb{R}_{+}^{n+1}.$
We denote $\overrightarrow{w}=(w_1,\cdots, w_m),$ where $\omega_i,$ $i=1,2,\cdots,m$ be weights in $\mathbb{R}^{n}.$
We say that
$(\mu,\overrightarrow w)$ satisfies the $B'_{\overrightarrow{P}}$ condition if
\begin{equation*}\label{Bp_constant}
    [\mu,\overrightarrow w]_{B'_{\overrightarrow P}}:= \sup_Q \big(\frac{\mu(\widetilde{Q})}{|Q|}\big)^{\frac{1}{p}} \prod^m_{i=1}\frac{w_i(Q)}{|Q|}
  \exp\Big(\frac{1}{|Q|}\int_Q\log\prod^m_{i=1}w_i^{-\frac{1}{p_i}} dx\Big)<\infty.
\end{equation*}
where the supremum is taken over all cubes in $\mathbb{R}^n.$
\end{definition}

Last, let us recall the definitions of $RH_{\overrightarrow{p}}$ and $W_{\overrightarrow P}^\infty$ which were introduced in \cite{Chen-Damian}.

\begin{definition}\cite{Chen-Damian}\label{defi_RH}
Let $\overrightarrow{P}=(p_1,\cdots,p_m)$ with $1< p_1,\cdots,p_m<\infty$ and $1/{p_1}+\cdots+1/{p_m}=1/p$.
We denote $\overrightarrow{w}=(w_1,\cdots, w_m),$ where $\omega_i,$ $i=1,2,\cdots,m$ be weights in $\mathbb{R}^{n}.$
We say that $\overrightarrow{\omega}$ satisfies the reverse H\"{o}lder's condition $RH_{\overrightarrow{p}},$  if
$$\prod\limits^m_{i=1}\big(\int_Q\sigma_idx\big)^{\frac{p}{p_i}}
\leq  C\int_Q\prod\limits^m_{i=1}\sigma_i^{\frac{p}{p_i}}
dx,~\forall ~cube~Q,$$
where $\sigma_i=\omega_i^{1-p_i'},$ $i=1,\cdot\cdot\cdot,m$ and the smallest constant $C$ is denoted by $[\overrightarrow{\omega}]_{RH_{\overrightarrow{p}}}.$
\end{definition}

\begin{definition}\cite{Chen-Damian}\label{defi_W infty}
Let $\overrightarrow{P}=(p_1,\cdots,p_m)$ with $1< p_1,\cdots,p_m<\infty$ and $1/{p_1}+\cdots+1/{p_m}=1/p$.
We denote $\overrightarrow{w}=(w_1,\cdots, w_m),$ where $\omega_i,$ $i=1,2,\cdots,m$ be weights in $\mathbb{R}^{n}.$
We say that
 $\overrightarrow w$ satisfies the $W_{\overrightarrow P}^\infty$ condition if
    \begin{equation*}\label{Fujii_constant}
      [\overrightarrow w]_{W_{\overrightarrow P}^\infty}=\sup_Q \Big(\int_Q\prod^m_{i=1}M(w_i\chi_Q)^{\frac{p}{p_i}}dx\Big)\Big(\int_Q\prod^m_{i=1}w_i^{\frac{p}{p_i}} dx\Big)^{-1}<\infty,
    \end{equation*}
where the supremum is taken over all cubes in $\mathbb{R}^n.$
\end{definition}

\begin{remark} Let $\overrightarrow{P}=(p_1,\cdots,p_m)$ with $1< p_1,\cdots,p_m<\infty$ and $1/{p_1}+\cdots+1/{p_m}=1/p$.
If $\overrightarrow{\omega}$ satisfies the multilinear $A_{\overrightarrow{P}}$ condition,
then $\overrightarrow{\omega}\in RH_{\overrightarrow{p}}$
and $\overrightarrow{\omega}\in W_{\overrightarrow P}^\infty $ \cite[Proposition 2.3]{cao-xue}.
\end{remark}

\section{Proofs}\label{proofs}

\begin{proof}[Proof of Lemma \ref{Lemma:CZ}] Let
$\mathcal{M}^{\mathfrak{D}}(\overrightarrow{f})(x)=\sup_{x\in Q\in\mathfrak{D}}\prod_{i=1}^m\frac{1}{|Q|}\int_{Q}f_i(y)dy.$
It follows that $$|\{\mathcal{M}^{\mathfrak{D}}\overrightarrow{f}(x)>a^k|\leq\sum\limits_{i=1}^{m}|\{M^{\mathfrak{D}}f_i(x)>a^{k/m}|<\infty,$$
where $M^{\mathfrak{D}}(f_i)(x)=\sup_{x\in Q\in\mathfrak{D}}\frac{1}{|Q|}\int_{Q}f_i(y)dy,~i=1,2,\cdot\cdot\cdot,m.$
For each $(x,t)\in \widetilde{\Omega}_{k},$ there is a maximal cube $Q\in\mathfrak{D}$ such that $x\in Q,$ $l(Q)\geq t$ and
$$\prod\limits_{i=1}^{m}\frac{1}{|Q|}\int_{Q}f_i(x)dx>a^k,$$
for otherwise $\{\mathcal{M}^{\mathfrak{D}}\overrightarrow{f}(x)>a^k\}$ would have infinite measure.
It's clear that this collection of cubes satisfies conditions \eqref{lemde1}, \eqref{lemde2} and \eqref{lemde3}.
To end the proof we have to prove that the family $\bigcup\limits_{k,j}\{Q_{j}^{k}\}$ is sparse.

For each $k\in\mathbb{Z},$ we observe that
$$
|Q_{j}^{k}\cap \Omega_{k+1}|=\sum_{j'\in J_{k+1}}|Q_{j}^{k}\cap Q_{j'}^{k+1}|
$$
Now since $Q_{j}^{k},Q_{j'}^{k+1}\in\mathfrak{D},$ we have
that $Q_{j}^{k}\cap Q_{j'}^{k+1}\not=\emptyset$ implies that either
$Q_{j}^{k}\subseteq Q_{j'}^{k+1}$ or $Q_{j'}^{k+1}\subseteq Q_{j}^{k}$.
Now we observe that from the definition of $\Omega_{k}$ it follows that
$Q_{j'}^{k+1}\subseteq \Omega_{k}$. Consequently $Q_{j'}^{k+1}\subseteq Q_{j}^{k}$
by maximality. Taking that into account, we have
\begin{eqnarray*}
|Q_{j}^{k}\cap \Omega_{k+1}| &=&\sum_{j'\in J_{k+1}}|Q_{j}^{k}\cap Q_{j'}^{k+1}|=\sum_{Q_{j'}^{k+1}\subseteq Q_{j}^{k}}|Q_{j'}^{k+1}|.
\end{eqnarray*}
It follows from \eqref{lemde1} and H\"{o}lder's inequality that
\begin{eqnarray*}
|Q_{j}^{k}\cap \Omega_{k+1}|
&\leq&\sum_{Q_{j'}^{k+1}\subseteq Q_{j}^{k}}(\frac{1}{a^{k+1}})^{\frac{1}{m}}(\prod\limits_{i=1}^{m}\int_{Q_{j'}^{k+1}}f_i(y)dy)^{\frac{1}{m}}\\
&\leq&(\frac{1}{a^{k+1}})^{\frac{1}{m}}(\prod\limits_{i=1}^{m}\int_{Q_{j}^{k}}f_i(y)dy)^{\frac{1}{m}}\\
&\leq&(\frac{a^k2^{mn}}{a^{k+1}})^{\frac{1}{m}}|Q_{j}^{k}|=\frac{1}{2}|Q_{j}^{k}|.
\end{eqnarray*}
\end{proof}

\begin{proof}[Proof of Theorem \ref{theorem_Ap}] We shall follow the scheme: $\eqref{theorem_Ap_3}\Leftrightarrow\eqref{theorem_Ap_1},$ $\eqref{theorem_Ap_2}\Leftrightarrow\eqref{theorem_Ap_21}$ and $\eqref{theorem_Ap_2}\Rightarrow\eqref{theorem_Ap_1}\Rightarrow\eqref{theorem_Ap_21}.$
Obviously, the equivalence  $\eqref{theorem_Ap_2}\Leftrightarrow\eqref{theorem_Ap_21}$ is trivial.

\eqref{theorem_Ap_3}$\Rightarrow$\eqref{theorem_Ap_1} For any cube $Q$ in $\mathbb{R}^n,$
it follows from H\"{o}lder's inequality and \eqref{theorem_Ap_3} that
\begin{eqnarray*}           &&\mu(\widetilde{Q})^{\frac{1}{p}}\prod\limits_{i=1}^{m}\left(\frac{1}{|Q|}\int_Q f_i(x)dx\right)\\
              &\leq&\mu(\widetilde{Q})^{\frac{1}{p}}\prod\limits_{i=1}^{m}\left(\frac{1}{|Q|}\int_Q f_i^{p_i}(x)\omega_i(x) dx\right)^{\frac{1}{p_i}}
                      \left(\frac{1}{|Q|}\int_Q\omega_i^{-\frac{p'_i}{p_i}}(x)dx\right)^{\frac{1}{p'_i}}\\
              &=&\prod\limits_{i=1}^{m}\left(\int_Q f_i^{p_i}(x)\omega_i(x) dx\right)^{\frac{1}{p_i}}\\
                      &\ &\left(\left(\frac{\mu(\widetilde{Q})}{|Q|}\right)^{\frac{1}{p}}\prod\limits_{i=1}^{m}\left(\frac{1}{|Q|}
                      \int_Q\omega_i^{-\frac{1}{p_i-1}}(x)dx\right)^{\frac{1}{p'_i}}\right)\\
              &\leq& [\mu,\overrightarrow{\omega}]_{A'_{\overrightarrow{p}}}\prod\limits^{m}_{i=1}\|f_i\chi_Q\|_{L^{p_i}(\mathbb{R}^n,\omega_i)}. \end{eqnarray*}
Then $[\mu,\overrightarrow{\omega}]'_{A'_{\overrightarrow{p}}}\leq[\mu,\overrightarrow{\omega}]_{A'_{\overrightarrow{p}}}.$

$\eqref{theorem_Ap_1}\Rightarrow\eqref{theorem_Ap_3}$ Let $Q$ be any cube in $\mathbb{R}^n.$ For $f_i=\omega_i^{-\frac{1}{p_i-1}}\chi_Q,$ we have
\begin{eqnarray*}
 &&\left(\frac{\mu(\widetilde{Q})}{|Q|}\right)^{\frac{1}{p}}\prod\limits_{i=1}^{m}\frac{1}{|Q|}
                      \int_Q\omega_i^{-\frac{1}{p_i-1}}(x)dx\\
       &\leq&[\mu,\overrightarrow{\omega}]'_{A'_{\overrightarrow{p}}}(\frac{1}{|Q|})^{\frac{1}{p}}\prod\limits_{i=1}^{m}
                 \left(\int_Q\omega_i^{-\frac{1}{p_i-1}}(x)dx\right)^{\frac{1}{p_i}}\\
       &=&[\mu,\overrightarrow{\omega}]'_{A'_{\overrightarrow{p}}}\prod\limits_{i=1}^{m}
                 \left(\frac{1}{|Q|}\int_Q\omega_i^{-\frac{1}{p_i-1}}(x)dx\right)^{\frac{1}{p_i}}.
\end{eqnarray*}
It follows that
\begin{eqnarray*}\left(\frac{\mu(\widetilde{Q})}{|Q|}\right)^{\frac{1}{p}}
                  \prod\limits_{i=1}^{m}\left(\frac{1}{|Q|}\int_Q\omega_i^{-\frac{1}{p_i-1}}(x)dx\right)^{\frac{1}{p'_i}}
        \leq [\mu,\overrightarrow{\omega}]'_{A'_{\overrightarrow{p}}}.\end{eqnarray*}
Then $[\mu,\overrightarrow{\omega}]_{A'_{\overrightarrow{p}}}\leq[\mu,\overrightarrow{\omega}]'_{A'_{\overrightarrow{p}}}.$

$\eqref{theorem_Ap_2}\Rightarrow\eqref{theorem_Ap_1}$ Let $Q$ be any cube in $\mathbb{R}^n.$ For $(x,t)\in \widetilde{Q},$ we have
$$\prod\limits_{i=1}^{m}\left(\frac{1}{|Q|}\int_Qf_i(y)dy\right)\leq\mathfrak{M}(\overrightarrow{f\chi_Q})(x,t).$$
It follows from $\eqref{theorem_Ap_2}$ that
\begin{multline*} \prod\limits_{i=1}^{m}\left(\frac{1}{|Q|}\int_Q f_i(y)dy\right)\mu(\widetilde{Q})^{\frac{1}{p}}\\
       \leq\lambda \mu\left(\{(x,t)\in\mathbb{R}_+^{n+1}:\mathfrak{M}(\overrightarrow{f\chi_Q})\geq\lambda\}\right)^{\frac{1}{p}}
       \leq \|\mathfrak{M}\|'\prod\limits^{m}_{i=1}\|f_i\chi_Q\|_{L^{p_i}(\mathbb{R}^n,\omega_i)},\end{multline*}
where $\displaystyle \lambda=\prod\limits_{i=1}^{m}\left(\frac{1}{|Q|}\int_Q f_i(y)dy\right).$

\eqref{theorem_Ap_1}$\Rightarrow$\eqref{theorem_Ap_21}
Without loss of generality, it suffices to prove that
$$
  \|\mathfrak{M}^{\mathfrak{D}}(\overrightarrow{f})\|_{L^{p,\infty}(\mathbb{R}_+^{n+1},\mu)}
  \lesssim[\mu,\overrightarrow{\omega}]'_{A'_{\overrightarrow{p}}}\prod_{i=1}^m \|f_i\|_{L^{p_i}(\mathbb{R}^n,\omega_i)}.
$$
for a general dyadic grid $\mathfrak{D}.$
Then
$\|\mathfrak{M}\|\lesssim[v,\overrightarrow{\omega}]'_{A'_{\overrightarrow{p}}}.$

Fixing $\lambda>0,$ let $k$ be the only integer such that $a^k\leq\lambda<a^{k+1},$ where $a=2^{m(n+1)}.$
It follows from Lemma \ref{Lemma:CZ} that
\begin{eqnarray*}
             \lambda^p \mu\big(\{(x,t)\in \mathbb{R}_+^{n+1}:\mathfrak{M}^{\mathfrak{D}}(\overrightarrow{f})>\lambda\}\big)
             &\leq&(a^{k+1})^p \mu\big(\{(x,t)\in \mathbb{R}_+^{n+1}:\mathfrak{M}^{\mathfrak{D}}(\overrightarrow{f})>a^k\}\big)\\
             &=&(a^{k+1})^p \mu\big(\bigcup_{j\in J_{k}}\widetilde{Q_{j}^{k}}\big)\\
             &=&a^p\sum\limits_{j\in J_{k}}a^{kp}\mu(\widetilde{Q_{j}^{k}})\\
             &\leq&a^p\sum\limits_{j\in J_{k}}\mu(\widetilde{Q_{j}^{k}})\Big(\prod\limits_{i=1}^{m}\frac{1}{|Q_{j}^{k}|}\int_{Q_{j}^{k}}f_i(y)dy\Big)^p.
\end{eqnarray*}
Using \eqref{theorem_Ap_1} and H\"{o}lder's inequality, we get
\begin{eqnarray*}
             &&\lambda^p \mu\big(\{(x,t)\in \mathbb{R}_+^{n+1}:\mathfrak{M}^{\mathfrak{D}}(\overrightarrow{f})>\lambda\}\big)\\
             &\leq&a^p([\mu,\overrightarrow{\omega}]'_{A'_{\overrightarrow{p}}})^p
                   \sum\limits_{j\in J_{k}}\prod\limits^{m}_{i=1}\|f_i\chi_{Q_{j}^{k}}\|_{L^{p_i}(\omega_i)}^p\\
             &\leq&a^p([\mu,\overrightarrow{\omega}]'_{A'_{\overrightarrow{p}}})^p\prod\limits^{m}_{i=1}
             \big(\sum\limits_{j\in J_{k}}\int_{Q_{j}^{k}}f_i^{p_i}(x)\omega_i(x)dx\big)^{\frac{p}{p_i}}\\
             &\leq&a^p\big([\mu,\overrightarrow{\omega}]'_{A'_{\overrightarrow{p}}}\prod\limits^{m}_{i=1}\|f_i\|_{L^{p_i}(\omega_i)}\big)^p.
\end{eqnarray*}
\end{proof}

\begin{proof}[Proof of Theorem \ref{thm:m}]
It suffices to prove that
$$
  \|\mathfrak{M}^{\mathfrak{D}}(\overrightarrow{f\sigma})\|_{L^p(\mathbb{R}_+^{n+1},\mu)}\lesssim [\mu,v_{\overrightarrow{w}}]^{1/p}_{ C_0}
[\overrightarrow{w}]^{\bar{p}}_{A'_{\overrightarrow P}}\prod_{i=1}^m \|f_i\|_{L^{p_i}(\mathbb{R}^n,\sigma_i)}.
$$
for a general dyadic grid $\mathfrak{D}$, and $\overrightarrow{f\sigma}=(f_1\sigma_1,\ldots,f_m\sigma_m).$

Without loss of generality, let $p_1=\min\{p_1,\cdots,p_m\}.$ Let $a=2^{m(n+1)}.$ It follows from Remark \ref{decop} that
\begin{eqnarray*}
\lefteqn{\int_{\mathbb{R}_+^{n+1}}\mathfrak{M}^{\mathfrak{D}}(\overrightarrow{f\sigma})^pd\mu \leq a^p\sum_{Q\in \mathfrak{S}} \prod_{i=1}^m\Big(\frac{1}{|Q|}\int_Q f_i(x)\sigma_i(x)dx\Big)^p\mu(\widetilde{Q})}\\
&=&a^p\sum_{Q\in \mathfrak{S}}\frac{\mu(\widetilde{Q})^{p_1'}\prod_{i=1}^m\sigma_i(Q)^{pp_1'/{p_i'}}}{|Q|^{mpp_1'}}
         \bigg( \prod_{i=1}^m \int_{Q}f_i(x)\sigma_i(x)dx \bigg)^p\\
&&\quad \frac{|Q|^{mp(p_1'-1)}}{\mu(\widetilde{Q})^{p_1'-1}\prod_{i=1}^m\sigma_i(Q)^{pp_1'/{p_i'}}}\\
&\leq& a^p[\overrightarrow{w}]_{A_{\overrightarrow{P}}}^{pp_1'}\sum_{Q\in \mathfrak{S}}\frac{2^{mp(p_1'-1)}|E_Q|^{mp(p_1'-1)}}{\mu(\widetilde{Q})^{p_1'-1}\prod_{i=1}^m\sigma_i(Q)^{pp_1'/{p_i'}}}\bigg( \prod_{i=1}^m \int_{Q}f_i(x)\sigma_i(x)dx \bigg)^p.
\end{eqnarray*}
By H\"{o}lder's inequality, we have
\begin{eqnarray*}
|E_Q|&=&\int_{E_Q}v_{\overrightarrow{w}}^{\frac{1}{mp}}(x)\sigma_1^{\frac{1}{mp_1'}}(x)\cdots\sigma_m^{\frac{1}{mp_m'}}(x)dx\label{eq:h} \\
&\le& v_{\overrightarrow{w}}(E_Q)^{\frac{1}{mp}}\sigma_1(E_Q)^{\frac{1}{mp_1'}}\cdots\sigma_m(E_Q)^{\frac{1}{mp_m'}}\nonumber.
\end{eqnarray*}
Therefore,
\[
  |E_Q|^{mp(p_1'-1)}\le v_{\overrightarrow{w}}(E_Q)^{p_1'-1}\sigma_1(E_Q)^{\frac{p(p_1'-1)}{p_1'}}\cdots\sigma_m(E_Q)^{\frac{p(p_1'-1)}{p_m'}}
\]
and
\[
  \frac{p(p_1'-1)}{p_i'}-\frac{p}{p_i}=\frac{pp_1'}{p_i'}-p\ge 0.
\]
Since $E_Q\subset Q$, we have
\[
  v_{\overrightarrow{w}}(E_Q)^{p_1'-1}\leq v_{\overrightarrow{w}}(Q)^{p_1'-1}\leq [\mu,v_{\overrightarrow{w}}]_{ C_0}\mu(\widetilde{Q})^{p_1'-1}
\]
and hence
\[
  \sigma_i(E_Q)^{\frac{p(p_1'-1)}{p_i'}-\frac{p}{p_i}}\le  \sigma_i(Q)^{\frac{pp_1'}{p_i'}-p},\quad i=1,\cdots,m.
\]
It follows that
\begin{eqnarray*}
\lefteqn{\sum_{Q\in \mathfrak{S}}\frac{|E_Q|^{mp(p_1'-1)}}{\mu(\widetilde{Q})^{p_1'-1}\prod_{i=1}^m\sigma_i(Q)^{pp_1'/{p_i'}}}
\bigg( \prod_{i=1}^m \int_{Q}f_i(x)\sigma_i(x)dx\bigg)^p}\\
&\leq&[\mu,v_{\overrightarrow{w}}]_{ C_0}\sum_{Q\in \mathfrak{S}}\prod_{i=1}^m
\bigg(  \frac{1}{\sigma_i(Q)}\int_{Q}f_i(x)\sigma_i(x)dx\bigg)^p \sigma_i(E_Q)^{p/{p_i}}\\
&\leq&[\mu,v_{\overrightarrow{w}}]_{ C_0}\prod_{i=1}^m\left( \sum_{Q\in \mathfrak{S}}\bigg(\frac{1}{\sigma_i(Q)}\int_{Q}f_i(x)\sigma_i(x)dx \bigg)^{p_i} \sigma_i(E_{Q})\right)^{p/{p_i}}\\
&\leq&[\mu,v_{\overrightarrow{w}}]_{ C_0}\prod_{i=1}^m\|M_{\sigma_i}^{\mathfrak{D}}(f_i)\|_{L^{p_i}(\mathbb{R}^n,\sigma_i)}^p\lesssim[\mu,v_{\overrightarrow{w}}]_{ C_0}\prod_{i=1}^m\|f_i\|_{L^{p_i}(\mathbb{R}^n,\sigma_i)}^p.
\end{eqnarray*}
Hence
$$
  \|\mathfrak{M}^{\mathfrak{D}}(\overrightarrow{f})\|_{L^p(\mu)}\lesssim [v_{\overrightarrow{w}}, \mu]_{ C_0}^{1/p} [\overrightarrow{w}]_{A_{\overrightarrow{P}}}^{{\bar{p}}}\prod_{i=1}^m \|f_i\|_{L^{p_i}(w_i)}.
$$
This completes the proof.
\end{proof}

\begin{proof}[Proof of Theorem \ref{prop s}]
Let $k\in Z.$ Using \ref{Lemma:CZ}, we have
\begin{eqnarray*}\mu\big(\{(x,t)\in \mathbb{R}_+^{n+1} : \mathfrak{M}^{\mathfrak{D}}\overrightarrow{f}(x,t)>a^k\}\big)
&=&\sum_{j\in J_{k}}\mu(\widetilde{Q_{j}^{k}}).\end{eqnarray*}
It follows from Definition \ref{def_C_infty} that
$$\mu\big(\{(x,t)\in \mathbb{R}_+^{n+1} : \mathfrak{M}^{\mathfrak{D}}\overrightarrow{f}(x,t)>a^k\}\big)
\leq[\mu,v_{\overrightarrow{w}}]_{C_{\infty}}\sum_{j\in J_{k}}v_{\overrightarrow{w}}(Q_{j}^{k}).$$
Because of $\sum_{j\in J_{k}}v_{\overrightarrow{w}}(Q_{j}^{k})=v_{\overrightarrow{w}}(\{\mathfrak{M}^{\mathfrak{D}}\overrightarrow{f}(x)>a^k\}),$ we obtain
\begin{equation}\label{eq mu}\mu\big(\{(x,t)\in \mathbb{R}_+^{n+1} : \mathfrak{M}^{\mathfrak{D}}\overrightarrow{f}(x,t)>a^k\}\big)
\leq[\mu,v_{\overrightarrow{w}}]_{C_{\infty}}v_{\overrightarrow{w}}\big(\{\mathcal{M}^{\mathfrak{D}}\overrightarrow{f}(x)>a^k\}\big).\end{equation}
Then
\begin{eqnarray*}
&&\int_{\mathbb{R}_+^{n+1}}\mathfrak{M}^{\mathfrak{D}}(\overrightarrow{f\sigma})^pd\mu\\
&=&p\int_0^{\infty}\lambda^{p-1}\mu\big(\{(x,t)\in \mathbb{R}_+^{n+1} : \mathfrak{M}^{\mathfrak{D}}\overrightarrow{f}(x,t)>\lambda\}\big)d\lambda\\
&=&p\sum\limits_{k\in \mathbb{Z}}\int_{a^k}^{a^{k+1}}\lambda^{p-1}\mu\big(\{(x,t)\in \mathbb{R}_+^{n+1} : \mathfrak{M}^{\mathfrak{D}}\overrightarrow{f}(x,t)>\lambda\}\big)d\lambda\\
&\leq&p\sum\limits_{k\in \mathbb{Z}}({a^{k+1}}-{a^k})a^{(k+1)(p-1)}\mu\big(\{(x,t)\in \mathbb{R}_+^{n+1} : \mathfrak{M}^{\mathfrak{D}}\overrightarrow{f}(x,t)>a^k\}\big).
\end{eqnarray*}
It follows from \eqref{eq mu}, we have
\begin{eqnarray*}
&&\int_{\mathbb{R}_+^{n+1}}\mathfrak{M}^{\mathfrak{D}}(\overrightarrow{f\sigma})^pd\mu\\
&\leq&p[\mu,v_{\overrightarrow{w}}]_{C_{\infty}}\sum\limits_{k\in \mathbb{Z}}({a^{k+1}}-{a^k})a^{(k+1)(p-1)}v_{\overrightarrow{w}}\big(\{\mathcal{M}^{\mathfrak{D}}\overrightarrow{f}(x)>a^k\}\big)\\
&=&a^{2p-1}p[\mu,v_{\overrightarrow{w}}]_{C_{\infty}}\sum\limits_{k\in \mathbb{Z}}({a^{k}}-{a^{k-1}})a^{(k-1)(p-1)}v_{\overrightarrow{w}}\big(\{\mathcal{M}^{\mathfrak{D}}\overrightarrow{f}(x)>a^k\}\big)\\
&\leq&a^{2p-1}p[\mu,v_{\overrightarrow{w}}]_{C_{\infty}}\sum\limits_{k\in \mathbb{Z}}\int_{a^k}^{a^{k+1}}\lambda^{p-1}v_{\overrightarrow{w}}\big(\{\mathcal{M}^{\mathfrak{D}}\overrightarrow{f}(x)>\lambda\}\big)d\lambda\\
&=&a^{2p-1}[\mu,v_{\overrightarrow{w}}]_{C_{\infty}}\int_{\mathbb{R}^n}\mathcal{M}^{\mathfrak{D}}(\overrightarrow{f})^pv_{\overrightarrow{w}}dx.
\end{eqnarray*}
Recalling that\cite[Theorem 1.2]{Li Moen Sun}
$$ \|\mathcal{M}\|_{\prod_{i=1}^m {L^{p_i}(\mathbb{R}^n,w_i)}\rightarrow L^p(\mathbb{R}^{n},v_{\overrightarrow{w}})}\lesssim
[\overrightarrow{w}]^{\bar{p}}_{A_{\overrightarrow P}},$$ we have
\begin{equation*}
  \|\mathfrak{M}(\overrightarrow{f})\|_{L^p(\mathbb{R}_+^{n+1},\mu)}\lesssim [\mu,v_{\overrightarrow{w}}]^{1/p}_{ C_{\infty}}
[\overrightarrow{w}]^{\bar{p}}_{A_{\overrightarrow P}}\prod_{i=1}^m \|f_i\|_{L^{p_i}(\mathbb{R}^n,w_i)}.
\end{equation*}
\end{proof}

\begin{proof}[Proof of Theorem \ref{theorem_Sp}] It is clear that $(1)\Leftrightarrow(2)\Rightarrow(3)$
without $(v,\overrightarrow{\omega})\in RH_{\overrightarrow{p}},$
so we omit them.

$(3)\Rightarrow(2)$ By Remark \ref{decop}, it
suffices to prove
\begin{equation*}
\|\mathfrak{M}^{\mathfrak{D}}(\overrightarrow{f\sigma})\|_{L^p(\mathbb{R}_+^{n+1},\mu)}
\lesssim[\overrightarrow{\omega}]_{RH_{\overrightarrow{p}}}^{\frac{1}{p}}[v,\overrightarrow{\omega}]_{S_{\overrightarrow{p}}}
\prod\limits^m_{i=1}\|f_i\|_{L^{p_i}(\mathbb{R}^n,\sigma_i)},
\end{equation*}
where $\mathfrak{D}$ is a general dyadic grid. Let $a=2^{m(n+1)}.$ It follows that
\begin{eqnarray*}
   & &\int_{R_+^{n+1}}\mathfrak{M}^{\mathfrak{D}}(\overrightarrow{f\sigma})^p d\mu\\
   &\leq&a^p\sum\limits_{Q\in\mathfrak{S}} \mu(\widehat{E_Q})\big(\prod\limits_{i=1}
      \limits^{m}\frac{1}{|Q|}\int_{Q}f_i(y_i)\sigma_i(y_i)dy_i\big)^p\\
   &=&a^p\sum\limits_{Q\in\mathfrak{S}}\Big(\mu(\widehat{E_Q})\big(\prod\limits_{i=1}
      \limits^{m}\frac{\sigma_i(Q)}{|Q|}\big)^p\Big)\big(\prod\limits_{i=1}
      \limits^{m}\frac{1}{\sigma_i(Q)}\int_{Q}f_i(y_i)\sigma_i(y_i)dy_i\big)^p\\
   &=&a^p\sum\limits_{Q\in \mathfrak{D}}a_Q\big(\prod\limits_{i=1}
      \limits^{m}\frac{1}{\sigma_i(Q)}\int_{Q}f_i(y_i)\sigma_i(y_i)d y_i\big)^p,
\end{eqnarray*}
where
$$
a_Q:=
\begin{cases}
 \mu(\widehat{E_Q})\big(\prod\limits_{i=1}
      \limits^{m}\frac{\sigma_i(Q)}{|Q|}\big)^p, &\hbox{ if } Q\in \mathfrak{S};\\
0,  &\hbox{ if else.}
\end{cases}
$$

Now, we check the assumption (\ref{lem_Carleson_ass}) of Lemma \ref{lem_Carleson_thm}. For $R\in \mathfrak{D},$ we have
\begin{eqnarray}\sum\limits_{Q\subseteq R}a_Q&=&\sum\limits_{Q\in \mathfrak{S},Q\subseteq R}a_Q\nonumber\\
&=&\sum\limits_{Q\in \mathfrak{S},Q\subseteq R}\mu(\widehat{E_Q})\big(\prod\limits_{i=1}
      \limits^{m}\frac{\sigma_i(Q)}{|Q|}\big)^p\nonumber\\
&=&\sum\limits_{Q\in \mathfrak{S},Q\subseteq R}\int_{\widehat{E_Q}}\big(\prod\limits_{i=1}
      \limits^{m}\frac{\sigma_i(Q)}{|Q|}\big)^p d\mu\nonumber.\end{eqnarray}
It follows that
\begin{eqnarray}\sum\limits_{Q\subseteq R}a_Q&\leq&\sum\limits_{Q\in \mathfrak{S},Q\subseteq R}\int_{\widehat{E_Q}}
      \big(\mathfrak{M}^{\mathfrak{D}}(\overrightarrow{\sigma\chi_R})\big)^p(x,t) d\mu\nonumber\\
&\leq&\int_{\widetilde{R}}\big(\mathfrak{M}^{\mathfrak{D}}(\overrightarrow{\sigma\chi_R})\big)^p (x,t)d\mu\nonumber\\
&\leq&[\mu,\overrightarrow{\omega}]_{S'_{\overrightarrow{p}}}^p
\prod\limits^m_{i=1}|R|^{\frac{p}{p_i}}_{\sigma_i}\label{by_Sp_a}\\
&\leq&[\mu,\overrightarrow{\omega}]_{S'_{\overrightarrow{p}}}^p
      [\overrightarrow{\omega}]_{RH_{\overrightarrow{p}}}\int_R
      \prod\limits^m_{i=1}\sigma_i^{\frac{p}{p_i}}(x)dx\label{by_Rh_a}
\end{eqnarray} where conditions $S¡®_{\overrightarrow{p}}$ and $RH_{\overrightarrow{p}}$
are used in (\ref{by_Sp_a}) and (\ref{by_Rh_a}), respectively.
Thus, we obtain (\ref{lem_Carleson_ass}).
It follows from Lemma \ref{lem_Carleson_thm} that\begin{equation*}
\|\mathfrak{M}^{\mathfrak{D}}(\overrightarrow{f\sigma})\|_{L^p(\mathbb{R}_+^{n+1},\mu)}
\lesssim[\overrightarrow{\omega}]_{RH_{\overrightarrow{p}}}^{\frac{1}{p}}[v,\overrightarrow{\omega}]_{S¡¯_{\overrightarrow{p}}}
\prod\limits^m_{i=1}\|f_i\|_{L^{p_i}(\mathbb{R}^n,\sigma_i)}.
\end{equation*}
\end{proof}

\begin{proof}[Proof of Theorem \ref{theorem_Bp}] Because of $(1)\Longleftrightarrow(2),$ we only prove that $(2)$ is valid.
As we discussed in Theorem \ref{theorem_Sp}, we have
\begin{eqnarray}
   &~&\int_{\mathbb{R}^{n+1}_+}\mathfrak{M}^{\mathfrak{D}}(\overrightarrow{f\sigma})^p d\mu
         \leq a^p\sum\limits_{Q\in\mathfrak{S}}a^{kp}\mu(\widetilde{Q})\nonumber\\
   &\leq&a^p\sum\limits_{Q\in\mathfrak{S}} \mu(\widetilde{Q})\big(\prod\limits_{i=1}
      \limits^{m}\frac{1}{|Q|}\int_{Q}f_i(x)\sigma_i(x)dx\big)^p\nonumber\\
   &=&a^p\sum\limits_{Q\in\mathfrak{S}}\Big(\mu(\widetilde{Q})\big(\prod\limits_{i=1}
      \limits^{m}\frac{\sigma_i(Q)}{|Q|}\big)^p\Big)\big(\prod\limits_{i=1}
      \limits^{m}\frac{1}{\sigma_i(Q)}\int_{Q}f_i(x)\sigma_i(x)dx\big)^p\nonumber\\
   &=&a^p\sum\limits_{Q\in  \mathfrak{D}}a_Q
      \big(\prod\limits_{i=1}\limits^{m}\frac{1}{\sigma_i(Q)}\int_{Q}f_i(x)\sigma_i(x)dx\big)^p,\nonumber
   \end{eqnarray}
where
$$
a_Q:=
\begin{cases}
 \mu(\widetilde{Q})\big(\prod\limits_{i=1}
      \limits^{m}\frac{\sigma_i(Q)}{|Q|}\big)^p, &\hbox{ if } Q\in\mathfrak{S};\\
0,  &\hbox{ if else.}
\end{cases}
$$

Now, we check the assumption (\ref{lem_Carleson_ass}) of Lemma \ref{lem_Carleson_thm}.
Let $R\in \mathfrak{D},$ we have
\begin{eqnarray*}\sum\limits_{Q\subseteq R}a_Q&=&\sum\limits_{Q\in \mathfrak{S},Q\subseteq R}a_Q\\
&=&\sum\limits_{Q\in \mathfrak{S},Q\subseteq R}\mu(\widetilde{Q})\big(\prod\limits_{i=1}
      \limits^{m}\frac{\sigma_i(Q)}{|Q|}\big)^p\\
&\leq&[\mu,\overrightarrow{\omega}]^p_{B'_{\overrightarrow{p}}}
      \sum\limits_{Q\in \mathfrak{S},Q\subseteq R}
      \Big(|Q|\exp\big(\frac{1}{|Q|}\int_{Q}\ln\prod\limits^m_{i=1}\sigma_i^{\frac{p}{p_i}}(x)dx\big)\Big),
\end{eqnarray*}
where condition $B'_{\overrightarrow{p}}$
is used. Then
\begin{eqnarray*}\sum\limits_{Q\subseteq R}a_Q
&\leq&2[\mu,\overrightarrow{\omega}]^p_{B'_{\overrightarrow{p}}}\sum\limits_{Q\in \mathfrak{S},Q\subseteq R}|E_Q|\exp\big(\frac{1}{|Q|}
      \int_{Q}\ln\prod\limits^m_{i=1}\sigma_i^{\frac{p}{p_i}}(x)dx\big)\\
&\leq&2[\mu,\overrightarrow{\omega}]^p_{B'_{\overrightarrow{p}}}\sum\limits_{Q\in \mathfrak{S},Q\subseteq R}
      \int_{E_Q}G(\prod\limits^m_{i=1}\sigma_i^{\frac{p}{p_i}}\chi_R)(x)dx\\
&\leq&2[\mu,\overrightarrow{\omega}]^p_{B'_{\overrightarrow{p}}}
      \int_{\mathbb{R}^n}G(\prod\limits^m_{i=1}\sigma_i^{\frac{p}{p_i}}\chi_R)(x)dx.\end{eqnarray*}
It follows from the boundedness of $G$ (see \cite{T. Hytonen}) that
$$\sum\limits_{Q\subseteq R}a_Q\leq 2e[\mu,\overrightarrow{\omega}]^p_{B'_{\overrightarrow{p}}}\int_{R}\prod\limits^m_{i=1}\sigma_i^{\frac{p}{p_i}}(x)dx.$$
Using Lemma \ref{lem_Carleson_thm} and Remark \ref{decop}, we get
\begin{equation*}
\|\mathfrak{M}(\overrightarrow{f\sigma})\|_{L^p(\mathbb{R}_+^{n+1},\mu)}\lesssim
[\mu,\overrightarrow{\omega}]_{B'_{\overrightarrow{p}}}
\prod\limits^m_{i=1}\|f_i\|_{L^{p_i}(\mathbb{R}^n,\sigma_i)}.
\end{equation*}
\end{proof}

\begin{proof} [Proof of Theorem \ref{theorem_bi A Fujii}]
This proof is similar to one of Theorem \ref{theorem_Bp}.
For $\{a_Q\}_{Q\in\mathfrak{D}}$ defined in the proof of Theorem \ref{theorem_Bp}, it suffices to check that
$$\sum\limits_{Q\subseteq R}a_Q\leq 2[\mu,\overrightarrow{\omega}]^p_{A'_{\overrightarrow{p}}}
[\overrightarrow{\omega}]_{W_{\overrightarrow{p}}^\infty} \int_{R}\prod\limits^m_{i=1}\sigma_i(x)dx,~R\in\mathfrak{D}.$$
Indeed, for $R\in\mathfrak{D},$ it follows from the definitions of $A'_{\overrightarrow{p}}$ that
\begin{eqnarray*}\sum\limits_{Q\subseteq R}a_Q
&=&\sum\limits_{Q\in \mathfrak{S},Q\subseteq R}\mu(\widetilde{Q})\big(\prod\limits_{i=1}
      \limits^{m}\frac{\sigma_i(Q)}{|Q|}\big)^p\\
&\leq&[\mu,\overrightarrow{\omega}]^p_{A'_{\overrightarrow{p}}}
      \sum\limits_{Q\in \mathfrak{S},Q\subseteq R}\prod\limits_{i=1}
      \limits^{m}\big(\sigma_i(Q)\big)^{\frac{p}{p_i}}\\
&=&[\mu,\overrightarrow{\omega}]^p_{A'_{\overrightarrow{p}}}\sum\limits_{Q\in \mathfrak{S},Q\subseteq R}\Big(\prod\limits_{i=1}
      \limits^{m}\big(\frac{\sigma_i(Q)}{|Q|}\big)^{\frac{p}{p_i}}\Big)|Q|.
\end{eqnarray*}
It follows that
\begin{eqnarray*}\sum\limits_{Q\subseteq R}a_Q
&\leq&2[\mu,\overrightarrow{\omega}]^p_{A'_{\overrightarrow{p}}}\sum\limits_{Q\in \mathfrak{S},Q\subseteq R}|E_Q|\Big(\prod\limits_{i=1}
      \limits^{m}\big(\frac{\sigma_i(Q)}{|Q|}\big)^{\frac{p}{p_i}}\Big)\\
&\leq&2[\mu,\overrightarrow{\omega}]^p_{A'_{\overrightarrow{p}}}\sum\limits_{Q\in \mathfrak{S},Q\subseteq R}
      \int_{E_Q}\prod\limits^m_{i=1}M(\sigma_i\chi_R)^{\frac{p}{p_i}}(x)dx\\
&\leq&2[\mu,\overrightarrow{\omega}]^p_{A'_{\overrightarrow{p}}}\int_{R}\prod\limits^m_{i=1}M(\sigma_i\chi_R)^{\frac{p}{p_i}}(x)dx\\
&\leq&2[\mu,\overrightarrow{\omega}]^p_{A'_{\overrightarrow{p}}}[\overrightarrow{\omega}]_{W_{\overrightarrow{p}}^\infty}
\int_{R}\prod\limits^m_{i=1}\sigma_i^{\frac{p}{p_i}}(x)dx.
\end{eqnarray*}
\end{proof}

\bigskip
\noindent{\bf Acknowledgements} Wei Chen is supported by
the Natural Science Foundation of Jiangsu Province (Grant No. BK20161326)
and the School Foundation of Yangzhou University (2016CXJ001).

\end{document}